\documentclass[11pt,twoside, leqno]{article}
\usepackage{amssymb,amsmath}
\usepackage{amsthm,,latexsym}
\usepackage{graphicx}
\usepackage{enumerate}

\numberwithin{equation}{section}

\newtheorem{lemma}{LEMMA}[section]
\newtheorem{thm}[lemma]{THEOREM}
\newtheorem*{corollary}{COROLLARY}
\newtheorem*{corollary1}{COROLLARY 1}
\newtheorem*{corollary2}{COROLLARY 2}

\newtheorem*{proposition}{PROPOSITION 1}

\theoremstyle{definition}
\newtheorem*{rmk}{Remark}
\newtheorem*{remark}{Remark}

\newtheorem*{defn}{Definition}

\long\def\symbolfootnote[#1]#2{\begingroup%
\def\thefootnote{\fnsymbol{footnote}}\footnote[#1]{#2}\endgroup}

\begin{document}
\newcommand{\ul}{\underline}
\newcommand{\be}{\begin{equation}}
\newcommand{\ee}{\end{equation}}
\newcommand{\ben}{\begin{enumerate}}
\newcommand{\een}{\end{enumerate}}

\long\def\symbolfootnote[#1]#2{\begingroup%
\def\thefootnote{\fnsymbol{footnote}}\footnote[#1]{#2}\endgroup}
\baselineskip = 16pt
\noindent

\title{A trio of Bernoulli relations, their implications for the Ramanujan polynomials and the zeta constants. }
\date{}

\author{M. C. Lettington (Cardiff)}
\maketitle
\begin{abstract}
We study the interplay between recurrences for zeta related functions at integer values, `Minor Corner Lattice' Toeplitz determinants and integer composition based sums. Our investigations touch on functional identities due to Ramanujan and Grosswald, the transcendence of the zeta function at odd integer values, the Li Criterion for the Riemann Hypothesis and pseudo-characteristic polynomials for zeta related functions. We begin with a recent result for $\zeta(2s)$ and some seemingly new Bernoulli relations, which we use to obtain a generalised Ramanujan polynomial and properties thereof.
\end{abstract}

\section{Introduction}
\symbolfootnote[0]{I would like to thank Professor M N Huxley for all his invaluable support
and guidance in this problem and Professor M W Coffey for his perceptive appraisal.\newline
2010 \emph{Mathematics Subject Classification}: 11B68, 11C20, 11J81, 11S05.\newline
\emph{Key words and phrases}: Ramanujan Polynomials, Bernoulli Relations, Zeta Constants.}
Let $B_0 =1$ and define the $s$-th Bernoulli number, $B_s$, and the $s$-th Bernoulli polynomial, $B_s(x)$, in the usual fashion \cite{concrete},
\cite{conway},
so that
\be
B_s=-\frac{1}{s+1}\sum_{k=0}^{s-1}\binom{s+1}{k}B_k,\qquad
B_s(x)=\sum_{k=0}^s\binom{s}{k}B_{s-k}x^k.
\label{eq:n01}
\ee
In recent papers \cite{mcl1}, \cite{mcl2} we showed that the Bernoulli numbers satisfy the recurrence relation
\be
2^{2s-1}B_{2s}=
\frac{s}{2s+1}-\frac{1}{2s+1}\sum_{k=1}^{s-1}\binom{2s+1}{2k}2^{2k-1}B_{2k},
\label{eq:n0}
\ee
and we applied the well-known Bernoulli-zeta even integer identity \cite{borwein}
\be
\zeta(2s)=\frac{(-1)^{s+1}2^{2s-1}\pi^{2s}B_{2s}}{(2s)!},
\label{eq:n001}
\ee
to yield the result
\be
\zeta(2s)
=(-1)^{s-1}\left (\frac{\pi^{2s}s}{(2s+1)!}+\sum_{k=1}^{s-1}\frac{(-1)^{s-k}\pi^{2k}}{(2k+1)!}\zeta{(2s-2k)}\right ).
\label{eq:1}
\ee
Our results are motivated partly by the relations of the title (given in Lemma 2.1) and partly by the connection with results obtained by Murty et al. concerning Ramanujan polynomials \cite{murty1}, \cite{gun}.

The odd-indexed Ramanujan polynomials are defined by
\be
R_{2s+1}(z)=\sum_{k=0}^{s+1}\frac{B_{2k}B_{2s+2-2k}}{(2k)!(2s+2-2k)!}z^{2k}.
\label{eq:R1}
\ee
They satisfy the functional (reciprocal polynomial) equation
\be
R_{2s+1}(z)=z^{2s+2}R_{2s+1}\left (\frac{1}{z}\right ),
\label{eq:R2}
\ee
and occur in Ramanujan's renowned identity involving the odd zeta constants
\[
\alpha^{-s}\left (\frac{1}{2}\zeta(2s+1)+\sum_{n=1}^{\infty}\frac{n^{-(2s+1)}}{e^{2\alpha n}-1}\right )=
\]
\be
\frac{1}{\beta^{s}}\left (\frac{1}{2}\zeta(2s+1)+\sum_{n=1}^{\infty}\frac{n^{-(2s+1)}}{e^{2\beta n}-1}\right )
-2^{2s}\sum_{k=0}^{s+1}(-1)^k\frac{B_{2k}B_{2s+2-2k}}{(2k)!(2s+2-2k)!}\alpha^{s+1-k}\beta^k,
\label{eq:R3}
\ee
where $\alpha, \beta >0$ and $\alpha\beta = \pi^2$. By definition $(\ref{eq:R1})$, the sum involving Bernoulli numbers in $(\ref{eq:R3})$ therefore
corresponds to the Ramanujan polynomial
\[
\alpha^{s+1}R_{2s+1}\left ( i\sqrt{\frac{\beta}{\alpha}}\right )=\alpha^{s+1}R_{2s+1}\left ( i\frac{\beta}{\pi}\right ).
\]
The following definition enables us to generalise the Ramanujan polynomials to include the even-indexed values $R_{2s}(z)$.
\begin{defn}
Let $B^*_s$ and $B^\prime_s$ be defined for $s\geq 2$ and $s\geq1$ respectively by the recurrences
\be
B^*_{s}=-\frac{1}{s+1}\sum_{k=0}^{s-1}\binom{s+1}{k}2^{k-s}B_k,\qquad
B^\prime_{s}=-\frac{1}{s+1}\sum_{k=0}^{s-1}\binom{s+1}{k}2^{-s}B_k,
\label{eq:n16}
\ee
with initial values $B^*_0=B^\prime_0=1$ and $B^*_1=\frac{1}{4}$.

For $r\geq0$, we define the \emph{generalised Ramanujan polynomial} $Q_r(z)$ such that
\be
Q_r(z)=\sum_{k=0}^{\left [(r+1)/2\right ]}\frac{B_{r+1-2k}^* B_{2k}^*}{(r+1-2k)!(2k)!}z^{2k}.
\label{eq:n165}
\ee
\end{defn}
\begin{thm}
With the definition of $Q_r(z)$ in $(\ref{eq:n165})$, then for $r=2s+1$ we have
\be
Q_{2s+1}(z)=R_{2s+1}(z).
\label{eq:R13}
\ee
When $r=2s$ is even we have
\be
Q_{2s}(z)=4z^{2s+2}\left (R_{2s+1}\left (\frac{1}{z}\right )-
R_{2s+1}\left (\frac{1}{2z}\right )\right )
=4\left (R_{2s+1}(z)-\frac{1}{2^{2s+2}}R_{2s+1}(2z)\right ),
\label{eq:R14}
\ee
and defining $R_{2s}\left (z\right )=Q_{2s}\left (z\right )$ we deduce the two-term reciprocal relation
\be
R_{2s}\left (z\right )-R_{2s}\left (\frac{z}{2}\right )=z^{2s+2}\left (R_{2s}\left (\frac{1}{z}\right )-R_{2s}\left (\frac{1}{2z}\right )\right ).
\label{eq:R145}
\ee
Hence we can maintain the notation developed by Murty et al. and speak of the even-indexed Ramanujan polynomials, $R_{2s}\left (z\right )$, as well as the odd-indexed Ramanujan polynomials $R_{2s+1}\left (z\right )$.
\end{thm}
\begin{corollary1} For every integer $s\geq 1$ we have
\be
R_{2s+1}(2)=R_{2s+1}(1)=-\frac{(2s+1)B_{2s+2}}{(2s+2)!},
\qquad R_{2s}(1)=\frac{-B_{2s+1}^*}{(2s)!}
\label{eq:R146}
\ee
so that
\be
R_{2s+1}\left (\frac{1}{2}\right )=\frac{1}{2^{2s+2}}R_{2s+1}(1),
\qquad R_{2s}\left (\frac{1}{2}\right )=\sum_{k=0}^{s}\frac{B_{2s+1-2k}^* B_{2k}^\prime}{(2s+1-2k)!(2k)!}=0,
\label{eq:R15}
\ee
and when $s=2s_1$ is even we have for the complex values
\be
R_{2s+1}\left (i\right )=0,\qquad R_{2s}\left (i\right )=R_{2s}\left (\frac{i}{2}\right ).
\label{eq:R16}
\ee
\end{corollary1}
\begin{corollary2}
For every integer $s\geq 1$ at least one of
\be
\zeta(4s-1),\qquad \sum_{n=1}^{\infty}\frac{1}{n^{4s-1}(e^{2\pi n}-1)}
\label{eq:R22}
\ee
is transcendental.

Similarly, for every integer $s\geq 1$ at least one of
\be
\zeta(4s+1),\qquad \sum_{n=1}^{\infty}\frac{1}{n^{4s+1}}\left (\frac{1}{e^{\pi  n }-1}-\frac{1}{2^{4s}(e^{4\pi  n }-1)}\right )
\label{eq:R23}
\ee
is transcendental.
\end{corollary2}
From the first relation in $(\ref{eq:R145})$ we obtain
\be
\sum_{k=0}^{s+1}\left (2^{2k}-1\right )\frac{B_{2s+2-2k}B_{2k}}{(2s+2-2k)!(2k)!}
=R_{2s+1}(2)-R_{2s+1}(1)=0,
\label{eq:R11}
\ee
which in turn relates to a quadratic recurrence relation for the even zeta numbers (stated in Theorem 1.3) similar to that discussed by Dilcher \cite{dilcher}. In Lemma 2.1 we derive the Bernoulli relations of the title which enables us prove Theorems 1.1 and 1.2. We mention in passing that $(\ref{eq:R11})$ implies that the odd-indexed Ramanujan polynomials have a root approaching $2$ (from above) as $s\rightarrow \infty$.

In this paper we also show that recurrence relations of the type depicted in $(\ref{eq:1})$ are closely linked to functions related to $\zeta(2s)$ as well as to the Li equivalence for the Riemann Hypothesis. These type of recurrence relations can be expressed in determinant form, and in Theorem 1.4, we give a restatement of the Li equivalence in terms of determinant properties of a square matrix.

Further results concern the existence of pseudo characteristic equations for $\zeta(2s)$ and related functions on the interval $[1,\infty)$, where the approximations are exact at the end points $s=1$ and $s=\infty$, taking approximate values in between. In music a related type of problem is encountered when considering \emph{open or natural tuning} versus \emph{equal temperament tuning}. A harmoniously acceptable but inexact solution is obtained by dividing the interval $[1,2]$, representing the octave, into twelfths, by defining the frequency ratio of two adjacent notes (an equally tempered semitone) to be $2^{1/12}$. The approximation then agrees at the end points of the octave but takes approximate values in between.

For $1/\zeta(s)$ (and again related functions thereof) we also give a pseudo characteristic equation with bounds for the accuracy of these approximations in Theorem 1.5.

We now introduce some more notation.
\begin{defn}[of functions related to $\zeta(s)$]Let
\be
\zeta(s)=\sum_{n=1}^\infty \frac{1}{n^s},\qquad
\eta(s)=\sum_{n=1}^\infty\frac{(-1)^{n-1}}{n^s}= \left(1-\frac{1}{2^{s-1}}\right)\zeta(s),
\label{eq:m1}
\ee
\be
\theta(s)=\sum_{n=0}^\infty \frac{1}{(2n+1)^s}= \left(1-\frac{1}{2^s}\right)\zeta(s),
\qquad  \phi(s)=\sum_{n=1}^\infty \frac{1}{(2n)^s}= \frac{1}{2^s}\zeta(s).
\label{eq:n1}
\ee
Then
\be
\zeta(s)=\theta(s)+\phi(s),\qquad \hbox{\rm and}\qquad \eta(s)=\theta(s)-\phi(s).
\label{eq:n3}
\ee
\end{defn}

Theorem 1.2 gives linear recurrence relations, similar to that in $(\ref{eq:1})$,
for the functions $\eta(2s)$, $\theta(2s)$ and $\phi(2s)$.
\begin{thm} We have
\be
\theta(2s)=(-1)^{s-1}\left (\frac{(2s-1)\pi^{2s}}{4(2s)!}+\sum_{k=1}^{s-1}\frac{(-1)^{s-k}\pi^{2k}}{2(2k)!}
\zeta{(2s-2k)}\right ),
\label{eq:n4}
\ee
and
\be
\phi(2s)=(-1)^{s-1}\left (\frac{(2s-1)\pi^{2s}}{4(2s+1)!}+\sum_{k=1}^{s-1}\frac{(-1)^{s-k}\pi^{2k}}{(2k+1)!2^{2s-2k}}\zeta{(2s-2k)}\right ).
\label{eq:n5}
\ee
\end{thm}

\begin{corollary}
We have
\be
\theta(2s)=(-1)^{s-1}\left (\frac{\pi^{2s}}{4(2s)!}+
\sum_{k=1}^{s-1}\frac{(-1)^{s-k}\pi^{2k}}{(2k+1)!}\theta{(2s-2k)}\right ),
\label{eq:n8}
\ee
\be
\phi(2s)=(-1)^{s-1}\left (\frac{(2s-1)\pi^{2s}}{4(2s+1)!}+\sum_{k=1}^{s-1}\frac{(-1)^{s-k}\pi^{2k}}{(2k+1)!}\phi{(2s-2k)}\right ),
\label{eq:n10}
\ee
and
\be
\eta(2s)=(-1)^{s-1}\left (\frac{\pi^{2s}}{2(2s+1)!}+\sum_{k=1}^{s-1}\frac{(-1)^{s-k}\pi^{2k}}{(2k+1)!}\eta{(2s-2k)}\right ).
\label{eq:n11}
\ee
\end{corollary}
The recurrence relation in $(\ref{eq:1})$ was originally deduced by studying determinants \cite{mnh} and the leading coefficients of the \emph{geometric polynomials} $b_q^{(s)}$ in $m$, defined for $r\geq 0$ and $q=1,\ldots m$, by the polynomial recurrence relations
\be
b_q^{(2r+1)}= \binom{m+r-q+1}{2r+1}-\sum_{k=0}^{r-1}\frac{1}{2r-2k+1}\binom{m+r-k}{2r-2k}b_q^{(2k+1)},
\label{eq:m41}
\ee
\be
b_q^{(2r)}= -\binom{m+r-q}{2r}-\sum_{k=0}^{r-1}\frac{1}{2r-2k+1}\binom{m+r-k}{2r-2k}b_q^{(2k)},
\label{eq:m415}
\ee
and also for $b_0^{(2r)}$ with $r\geq 0$ in $(\ref{eq:m415})$ with $b_0^{(0)}=m$.
When $q=m$ in $(\ref{eq:m41})$, the leading coefficients of the polynomials then follow the Dirichlet eta function recurrence relation given in $(\ref{eq:n11})$.
\begin{defn}
Corresponding to the three infinite-dimensional vectors
\[
{\bf h}=(h_1,h_2,h_3,\ldots),\qquad {\bf H}=(H_1,H_2,H_3,\ldots),\qquad{\bf G}=(G_1,G_2,G_3,\ldots),
\]
we define
\be
\Delta_s({\bf h})=(-1)^s\left |
\begin{array}{cccccc}
h_1 & 1 & 0 & 0 &  \ldots & 0 \\
h_2 & h_1 & 1 & 0 &  \ldots & 0 \\
h_3 & h_2 & h_1 & 1 &  \ldots & 0\\
\vdots & \vdots & \vdots & \vdots &  \ddots & \vdots\\
h_{s-1} & h_{s-2} & h_{s-3} & h_{s-4} & \ldots & 1\\
h_s & h_{s-1} & h_{s-2} & h_{s-3} & \ldots & h_1\\
\end{array}
\right |,
\label{eq:d1}
\ee
\vspace{0cm}
\be
\Psi_s({\bf h},{\bf H})=(-1)^s\left |
\begin{array}{cccccc}
H_1 & 1 & 0 & 0 &  \ldots & 0 \\
H_2 & h_1 & 1 & 0 &  \ldots & 0 \\
H_3 & h_2 & h_1 & 1 &  \ldots & 0\\
\vdots & \vdots & \vdots & \vdots &  \ddots & \vdots\\
H_{s-1} & h_{s-2} & h_{s-3} & h_{s-4} & \ldots & 1\\
H_s & h_{s-1} & h_{s-2} & h_{s-3} & \ldots & h_1\\
\end{array}
\right |,
\label{eq:d11}
\ee
\vspace{0cm}
\be
\Lambda_s({\bf h},{\bf H},{\bf G})=(-1)^s\left |
\begin{array}{cccccc}
H_1 & 1 & 0 & 0 &  \ldots & 0 \\
H_2 & h_1 & 1 & 0 &  \ldots & 0 \\
H_3 & h_2 & h_1 & 1 &  \ldots & 0\\
\vdots & \vdots & \vdots & \vdots &  \ddots & \vdots\\
H_{s-1} & h_{s-2} & h_{s-3} & h_{s-4} & \ldots & 1\\
H_s & G_{s-1} & G_{s-2} & G_{s-3} & \ldots & G_1\\
\end{array}
\right |.
\label{eq:d12}
\ee

\noindent
We refer to $\Delta_s({\bf h})$ as an $s\times s$ \emph{minor corner layered determinant}, or type 1 MCL determinant for short; to $\Psi_s({\bf h},{\bf H})$ as a \emph{half-weighted} $s\times s$ MCL determinant, or type 2 MCL determinant for short, and to $\Lambda_s({\bf h},{\bf H},{\bf G})$ as a \emph{fully-weighted} $s\times s$ MCL determinant, or type 3 MCL determinant for short. Furthermore, if $H_k=G_k$ for each $k=1,\ldots , s$ then we call $\Lambda({\bf h},{\bf H},{\bf H})$ a \emph{balanced fully-weighted MCL determinant}.
\end{defn}
The closed forms for $b_q^{(t)}$ given in $(\ref{eq:m41})$ and $(\ref{eq:m415})$, were originally obtained by studying
associated magic squares under matrix multiplication \cite{mcl1}.
We will see later in Lemma 3.1 that all recurrence relations of this type can be expressed as one of the
three types of \emph{Minor Corner Layered} determinants defined above.

We can now state Theorem 1.3, which expresses both $\theta(2s+2)$ and $\zeta(2s+2)$ as a quadratic recurrence relation, an ``integer composition'' based sum and as a type one MCL determinant.
\begin{thm}
Let ${\bf p}$ and ${\bf q}$ be the two infinite-dimensional vectors defined such that
\[
{\bf p}=2(\phi(2),-\phi(4),\phi(6),-\phi(8),\ldots),
\qquad {\bf q}=2(\zeta(2),-\zeta(4),\zeta(6),-\zeta(8),\ldots).
\]
Then
\[
\theta(2s+2)=2\sum_{k=0}^{s-1}\phi(2s-2k)\theta(2k+2)=\frac{(-1)^s\pi^2}{8}\Delta_s({\bf p})
\]
\be
=\frac{\pi^2}{8}\mathop{\mathop{\sum_{t=1}^s \sum_{d_i\geq 0}}_{d_1+d_2+\ldots +d_s=t}}_{d_1+2d_2+\ldots+sd_s=s} 2^t\binom{t}{d_1,d_2,\ldots,d_s}
\phi^{d_1}(2)\phi^{d_2}(4)\ldots\phi^{d_s}(2s),
\label{eq:n36}
\ee
and
\[
\zeta(2s+2)=\frac{2}{2^{2s+2}-1}\sum_{k=0}^{s-1}(2^{2k+2}-1)\zeta(2s-2k)\zeta(2k+2)
=\frac{(-1)^s\pi^2}{2}\Delta_s({\bf q})
\]
\be
=\frac{\pi^2/2}{2^{2s+2}-1}\mathop{\mathop{\sum_{t=1}^s \sum_{d_i\geq 0}}_{d_1+d_2+\ldots +d_s=t}}_{d_1+2d_2+\ldots+sd_s=s}
2^{t}\binom{t}{d_1,d_2,\ldots,d_s} \zeta^{d_1}(2)\zeta^{d_2}(4)\ldots\zeta^{d_s}(2s).
\label{eq:n37}
\ee
\end{thm}
\begin{remark}
Similar sums have been considered by Dilcher \cite{dilcher}, where for $N\geq 1$, he defines
the $S_N(n)$ such that
\be
S_N(n)=\sum_{d_1+d_2+\ldots +d_s=n\phantom{1}} \sum_{d_i\geq 0}
\binom{2n}{2d_1,2d_2,\ldots,2d_N} B_{2d_1}B_{2d_2}\ldots B_{2d_N},
\label{eq:Dil1}
\ee
and the sequence $r_k^{(N)}$ of rational numbers recursively by $r_0^{(N)}=1$,
\[
r_k^{(N+1)}=\frac{-1}{N}r_k^{(N)}+\frac{1}{4}r_{k-1}^{(N-1)},
\]
with $r_k^{(N)}=0$ for $k<0$. For $2n>N$, Dilcher then shows that
\begin{eqnarray*}
S_N(n)&=&\frac{(2n)!}{(2n-N)!}\sum_{k=0}^{(N-1)/2}r_k^{(N)}\frac{B_{2n-2k}}{2n-2k}\\
&=& \sum_{d_1+d_2+\ldots +d_s=n\phantom{1}} \sum_{d_i\geq 0} \zeta(2d_1)\zeta(2d_2)\ldots\zeta(2d_N);
\label{eq:Dil2}
\end{eqnarray*}
note that Dilcher's sum includes the non-elementary value $\zeta(0)=-\frac{1}{2}$ (see Titchmarsh \cite{titchmarsh}, equation $(2.4.3)$).

\end{remark}
\vspace{5 mm}
In Lemma 3.3 we show that the linear recurrence relations already stated for $\zeta(2s)$, $\eta(2s)$, $\theta(2s)$ and $\phi(2s)$,
in $(\ref{eq:1})$, $(\ref{eq:n8})$, $(\ref{eq:n10})$ and $(\ref{eq:n11})$ can also be expressed as MCL determinants and as ``integer composition'' based sums.

This seemingly fundamental link between the even zeta based constants, closed form recurrence relations, ``integer composition'' based sums and MCL determinants also extends to the \emph{Li equivalence} for the Riemann Hypothesis.

The Li equivalence relies on the non-negativity of a sequence of real numbers $\{\lambda_n\}_{n=1}^\infty$ determined from the Riemann xi function as follows. Let
\[
\xi(s)=s(s-1)\pi^{-s/2}\Gamma\left (\frac{s}{2}\right )\zeta(s),
\]
\[
\lambda_n=\frac{1}{(n-1)!}\frac{d^n}{ds^n}[s^{n-1}\log{\xi(s)}]_{s=1},
\]
and
\[
\varphi(z)=\xi\left (\frac{1}{1-z}\right )=1+\sum_{j=1}^\infty a_j z^j,
\]
for $|z|<1/4$. Then $\xi(s)$ satisfies the functional equation $\xi(s)=\xi(1-s)$.
By expressing $\lambda_n$ as a sum
over the non-trivial zeros of $\zeta(s)$ and utilising Jacobi theta functions, Li \cite{Li} shows that $a_j$ is a positive real number for every positive integer $j$; that
\be
\lambda_n=\sum_{t=1}^s\frac{(-1)^{t-1}}{t} \mathop{\mathop{\sum }_{1\leq k_1,\ldots,k_t\leq n}}_{k_1+k_2+\ldots +k_t=n} a_{k_1}\ldots a_{k_t},
\label{eq:li1}
\ee
and that the recurrence relation
\be
\lambda_n=n a_n-\sum_{j=1}^{n-1}\lambda_j a_{n-j}
\label{eq:li2}
\ee
holds for every positive integer $n$.

\begin{proposition}[Li Criterion]
A necessary and sufficient condition for the nontrivial zeros of the Riemann zeta function $\zeta(s)$ to lie on the
critical line is that $\lambda_n$ is non-negative for every positive integer $n$.
\end{proposition}
Li obtains a corresponding equivalence for the Dedekind zeta function $\zeta_k(s)$ of an algebraic number field $k$. There is an illuminating discussion of the Li Criterion in \cite{coffey}, \cite{coffey1}.

We can reword the Li Criterion (and similarly for an algebraic number field) using half-weighted MCL determinants.
\begin{thm}
With $\lambda_j$ and $a_j$ defined as in $(\ref{eq:li1})$ and $(\ref{eq:li2})$, let
${\bf a}$ and ${\bf A}$ be the two infinite-dimensional vectors defined by
\[
{\bf a}=(a_1,a_2,a_3,a_4,\ldots),
\qquad {\bf A}=(a_1,2a_2,3a_3,4a_4,\ldots).
\]
Let $L_n$ be the $n\times n$ matrix given by
\be
L_n=\left (
\begin{array}{cccccc}
-a_1 & 1 & 0 & 0 &  \ldots & 0 \\
-2a_2 & a_1 & 1 & 0 &  \ldots & 0 \\
-3a_3 & a_2 & a_1 & 1 &  \ldots & 0\\
\vdots & \vdots & \vdots & \vdots &  \ddots & \vdots\\
-(n-1)a_{n-1} & a_{n-2} & a_{n-3} & a_{n-4} & \ldots & 1\\
-n a_n & a_{n-1} & a_{n-2} & a_{n-3} & \ldots & a_1\\
\end{array}
\right ),
\label{eq:A2}
\ee
and define $M_n$ such that $M_n=(-1)^n|L_n|$. Then
\be
M_n=\lambda_n=\Psi({\bf a},{-\bf A})=-\Psi({\bf a},{\bf A}),
\label{eq:A1}
\ee
and a necessary and sufficient condition for the nontrivial zeros of the Riemann
zeta function to lie on the critical line is that $n\times n$ half-weighted MCL determinant $M_n$,
given in $(\ref{eq:A1})$ satisfies $M_n\geq 0$ for all $n=1,2,3,\ldots$.
\end{thm}
Theorem 1.5 examines some `pseudo-characteristic polynomials' that approximate $\zeta(s)$, $1/\zeta(s)$ and related functions.
\begin{defn}[of pseudo characteristic polynomials]Let
\be
p_{s}(x)=\sum_{k=1}^{s-1}\frac{(-1)^{k-1}\pi^{2k}}{(2k+1)!}x^{2k},\qquad
q_{s}(x)=\sum_{k=0}^{s-1}\frac{(-1)^{k}\pi^{2k}}{(2k+1)!}x^{2k},\qquad
\label{eq:33}
\ee
\be
z_{s}(x)=\frac{(-1)^{s-1}s\pi^{2s}}{(2s+1)!}+p_{s}(x),\qquad
t_{s}(x)=\frac{(-1)^{s-1}\pi^{2s}}{4(2s)!}+p_{s}(x),
\label{eq:33a}
\ee
\be
e_{s}(x)=\frac{(-1)^{s-1}\pi^{2s}}{2(2s+1)!}+p_{s}(x),\qquad
f_{s}(x)=\frac{(-1)^{s-1}(2s-1)\pi^{2s}}{4(2s+1)!}+p_{s}(x).
\label{eq:33b}
\ee
The polynomials above are all of a similar structure to that in $(\ref{eq:1})$.
\end{defn}

\begin{thm}
For positive integers $s$ the polynomials $z_s(x)$ and $q_s(x)$, evaluated either at $k=2s$ or $k=2s-1$, satisfy the following inequalities.

\noindent
For $s\geq 17$
\be
\zeta(k)-3\{\zeta(k)\}^2 \leq  z_{s}(\zeta(k))\leq\zeta(k).
\label{eq:65}
\ee
For $s\geq 38$
\be
\theta(k)-3\{\theta(k)\}^2 \leq t_{s}(\theta(k))\leq\theta(k).
\label{eq:66}
\ee
For $s\geq 34$
\be
\frac{1}{\zeta(k)}-\{\zeta(k)\}^3 \leq 1+q_{s}(\zeta(k)) \leq \frac{1}{\zeta(k)}+11\{\zeta(k)\}^3.
\label{eq:67}
\ee
For $s\geq 114$
\be
\frac{1}{\theta(k)}-\{\theta(k)\}^3\leq1+ q_{s}(\theta(k)) \leq \frac{1}{\theta(k)}+11\{\theta(k)\}^3.
\label{eq:68}
\ee
Here $\{\zeta(k)\}=\zeta(k)-1$ is the fractional part of $\zeta(k)$.
Similar results hold for $e_{s}(\eta(k))$, $f_{s}(\phi(k))$ and $1+q_s(\eta(k))$.
\end{thm}

\section{Bernoulli relations}
We now establish the Bernoulli relations of the title. Different relations of this type were obtained by Woon \cite{woon}.
\begin{lemma}[Bernoulli trio]
Let $B^\prime_s$ and $B^*_s$ be defined as in $(\ref{eq:n16})$. Then for natural number~$s$, the following three identities hold:
\begin{enumerate}
\item[{\rm (i)}]\[B^*_{2s}= B_{2s},\qquad B^*_{2s-1}= \left ( 1-\frac{1}{2^{2s}},
\right )\frac{2B_{2s}}{s}\quad.\]
\item[{\rm (ii)}]  \[B^\prime_{s}= \frac{B_{s}}{2^s},\]
\item[{\rm (iii)}] \[T(x)=\sum_{s=1}^\infty \frac{2^{2s} B^*_{2s-1}}{(2s-1)!}x^{2s-2}
=\left ( 1+ \sum_{s=1}^\infty \frac{2^{2s} B^\prime_{2s}}{(2s)!}x^{2s}\right )^{-1}  \]
\[
=8\sum_{s=1}^\infty (-1)^{s-1}\frac{\theta(2s)}{\pi^{2s}}x^{2s-2}
=\left ( 1+ 2\sum_{s=1}^\infty (-1)^{s-1}\frac{\phi(2s)}{\pi^{2s}}x^{2s}\right )^{-1}.
\]
\end{enumerate}
\end{lemma}
As an immediate consequence of the second identity in $(\rm i)$, we have
\be
\theta(2s)=\frac{(-1)^{s+1}2^{2s-3}\pi^{2s}B_{2s-1}^*}{(2s-1)!}.
\label{eq:R10}
\ee
The first few values of $B_s$ and $B^*_s$ are given in the table above.
\begin{table}
\setlength{\tabcolsep}{0.6\tabcolsep}
\[
\renewcommand{\arraystretch}{1.4}
\begin{tabular}{|c|c|c|c|c|c|c|c|c|c|c|c|c|c|}\hline
s & 0 & 1 & 2 & 3 & 4 & 5 & 6 & 7 & 8 & 9 & 10 & 11 & 12\\ \hline
 $B(s)$ & 1 & $\frac{-1}{2}$ & $\frac{1}{6}$ & 0 & $\frac{-1}{30}$ & 0 & $\frac{1}{42}$ & 0 & $\frac{-1}{30}$ & 0 & $\frac{5}{66}$ & 0 & $\frac{-691}{2730}$\\

$B^\star(s)$ & 1 & $\frac{1}{4}$ & $\frac{1}{6}$ & $\frac{-1}{32}$ & $\frac{-1}{30}$ & $\frac{1}{64}$ & $\frac{1}{42}$ & $\frac{-17}{1024}$ & $\frac{-1}{30}$ &
$\frac{31}{1024}$ & $\frac{5}{66}$ & $\frac{-691}{8192}$ & $\frac{-691}{2730}$\\ \hline
\end{tabular}
\]
\end{table}

\begin{proof}
Let the $s$-th Bernoulli number, $B_s$, and the $s$-th Bernoulli polynomial, $B_s(x)$, be defined as in $(\ref{eq:n01})$,
where $B_0 =1$.

The first expression in $(\rm i)$ follows directly from rearranging the identity in ($\ref{eq:n0}$). We have
\[
2^{2s-1}B_{2s}=
\frac{s}{2s+1}-\frac{1}{2s+1}\sum_{k=1}^{s-1}\binom{2s+1}{2k}2^{2k-1}B_{2k}
\]
\[
=
\frac{s}{2s+1}+\frac{1}{2s+1}\left ( \frac{1}{2}-\frac{(2s+1)}{2}\right )
-\frac{1}{2s+1}\sum_{k=0}^{2s-2}\binom{2s+1}{k}2^{k-1}B_{k}
\]
\[
=-\frac{1}{2s+1}\sum_{k=0}^{2s-1}\binom{2s+1}{k}2^{k-1}B_{k},
\]
so that
\[
B_{2s}=-\frac{1}{2s+1}\sum_{k=0}^{2s-1}\binom{2s+1}{k}2^{k-2s}B_{k}=B^{*}_{2s},
\]
as required.

To obtain the second part of $(\rm i)$ we consider $B_s(x)$ with $s>1$ and $x=\frac{1}{2}$. Then we have
\[
(2^{1-s}-1)B_s=B_s\left(\frac{1}{2}\right )=
\sum_{k=0}^{s}\binom{s}{k}\left (\frac{1}{2}\right )^{s-k}B_k,
\]
yielding
\[
(2^{1-2s}-1)B_{2s}=
\sum_{k=0}^{2s}\binom{2s}{k}\left (\frac{1}{2}\right )^{2s-k}B_k.
\]
Thus we get
\[
\frac{(2^{2s-1}-1)}{2^{2s-1}}B_{2s}=
-\sum_{k=0}^{2s}\binom{2s}{k}2^{k-2s}B_k,
\]
\[
\frac{(2^{2s-1}-1)}{2^{2s-1}}B_{2s}+B_{2s}=
-\sum_{k=0}^{2s-2}\binom{2s}{k}2^{k-2s}B_k,
\]
so that
\[
(2^{2s}-1)B_{2s}=
-\sum_{k=0}^{2s-2}\binom{2s}{k}2^{k-1}B_k.
\]
It therefore follows that
\[
\left ( 1-\frac{1}{2^{2s}}\right )\frac{2B_{2s}}{s}=-\frac{2}{2^{2s}s}\sum_{k=0}^{2s-2}\binom{2s}{k}2^{k-1}B_k
\]
\[
=-\frac{1}{2s}\sum_{k=0}^{2s-2}\binom{2s}{k}2^{k-2s+1}B_k,
\]
and replacing $s$ with $2s-1$ in $(\ref{eq:n16})$ we deduce the result. The identity $(\ref{eq:R10})$ then follows from the definition of $\theta(2s)$.

Part $(\rm ii)$ of the Lemma can be obtained by simply multiplying through by~$2^{-s}$ in the definition for $B_s$, although for part $(\rm iii)$, we need to consider the series expansions of both $\coth{x}$ and $\tanh{x}$. It is known from \cite{handbook} that
\[
\coth{x}=x^{-1}+\sum_{s=1}^\infty \frac{2^{2s}B_{2s}}{(2s)!}x^{2s-1},\qquad |x|<\pi,
\]
and that
\[
\tanh{x}=\sum_{s=1}^\infty \frac{2^{2s}(2^{2s}-1)B_{2s}}{(2s)!}x^{2s-1},\qquad |x|<\frac{\pi}{2}.
\]
Writing
\[
T(x)=\frac{2}{x}\tanh{\frac{x}{2}}=\left (\frac{x}{2}\coth{\frac{x}{2}}\right )^{-1},\qquad |x|<\pi,
\]
then gives
\[
\sum_{s=1}^\infty \frac{2^{2s+1}(2^{2s}-1)B_{2s}}{(2s)!}\frac{x^{2s-2}}{2^{2s-1}}=\left (1+\sum_{s=1}^\infty
\frac{2^{2s}B_{2s}}{(2s)!}\frac{x^{2s}}{2^{2s}}\right )^{-1},
\]
so that
\[
\sum_{s=1}^\infty \frac{4(2^{2s}-1)B_{2s}}{(2s)!}x^{2s-2}=\left (1+\sum_{s=1}^\infty \frac{B_{2s}}{(2s)!}x^{2s}\right )^{-1},
\]
and
\[
\sum_{s=1}^\infty \frac{2}{s}\frac{(2^{2s}-1)B_{2s}}{(2s-1)!}x^{2s-2}=\left (1+\sum_{s=1}^\infty \frac{B_{2s}}{(2s)!}x^{2s}\right )^{-1}.
\]
We then apply the first two parts of this lemma to obtain the polynomial result
\be
T(x)=\sum_{s=1}^\infty \frac{2^{2s} B^*_{2s-1}}{(2s-1)!}x^{2s-2}
=\left ( 1+ \sum_{s=1}^\infty \frac{2^{2s} B^\prime_{2s}}{(2s)!}x^{2s}\right )^{-1},
\label{eq:n25}
\ee
and from the definitions of $\theta(2s)$ and $\phi(2s)$ we obtain the final display in part~$(\rm iii)$.
\end{proof}
With the aid of Lemma 2.1, we are now in a position to prove Theorems 1.1 and 1.2.

\begin{proof}[Proof of Theorem 1.1]
To see $(\ref{eq:R13})$ we have
\[
Q_{2s+1}(z)=\sum_{k=0}^{\left [\frac{2s+2}{2}\right ]}
\frac{B_{2s+2-2k}^* B_{2k}^*}{(2s+2-2k)!(2k)!}z^{2k}
=\sum_{k=0}^{s+1}\frac{B_{2s+2-2k} B_{2k}}{(2s+2-2k)!(2k)!}z^{2k}=R_{2s+1}(z),
\]
and for $(\ref{eq:R14})$, when $r=2s$,
\[
Q_{2s}(z)=R_{2s}(z)=\sum_{k=0}^{s}
\frac{B_{2s+1-2k}^* B_{2k}^*}{(2s+1-2k)!(2k)!}z^{2k}=\sum_{k=1}^{s+1}
\frac{B_{2s+2-2k} B_{2k-1}^*}{(2s+2-2k)!(2k-1)!}z^{2s+2-2k}
\]
\[
=\sum_{k=1}^{s+1}\frac{B_{2s+2-2k}}{(2s+2-2k)!(2k-1)!}\left ( 1-\frac{1}{2^{2k}}
\right )\frac{2B_{2k}}{k}z^{2s+2-2k}
\]
\[
=\sum_{k=1}^{s+1}\left (2^{2k}-1\right )\frac{4B_{2s+2-2k}B_{2k}}{(2s+2-2k)!(2k)!2^{2k}}z^{2s+2-2k}
\]
\[
=4z^{2s+2}\sum_{k=0}^{s+1}\left (2^{2k}-1\right )\frac{B_{2s+2-2k}B_{2k}}{(2s+2-2k)!(2k)!}
\left (\frac{1}{2z}\right )^{2k}
\]
\[
=4z^{2s+2}\left (R_{2s+1}\left (\frac{1}{z}\right )-
R_{2s+1}\left (\frac{1}{2z}\right )\right ).
\]
Using the odd-indexed reciprocal relationship of $(\ref{eq:R2})$ then gives
\[
z^{2s+2}R_{2s+1}\left (\frac{1}{z}\right )=R_{2s+1}(z),\qquad
z^{2s+2}R_{2s+1}\left (\frac{1}{2z}\right )=\frac{1}{2^{2s+2}}R_{2s+1}(2z),
\]
from which we obtain $Q_{2s}(z)=$
\[
R_{2s}(z)=4z^{2s+2}\left (R_{2s+1}\left (\frac{1}{z}\right )-
R_{2s+1}\left (\frac{1}{2z}\right )\right )
=4\left (R_{2s+1}(z)-\frac{1}{2^{2s+2}}R_{2s+1}(2z)\right ).
\]
which is the relationship given in $(\ref{eq:R14})$. Setting
\[
P_{2s}(z)=R_{2s}\left (z\right )-R_{2s}\left (\frac{z}{2}\right ),
\]
and then applying $(\ref{eq:R2})$ and $(\ref{eq:R14})$ we deduce the final statement of the theorem
\[
P_{2s}(z)=z^{2s+2}P_{2s}\left (\frac{1}{z}\right )=z^{2s+2}\left (R_{2s}\left (\frac{1}{z}\right )-R_{2s}\left (\frac{1}{2z}\right)\right ).
\]
It was proven in \cite{murty1} that
\[
R_{2s+1}(2)=-\frac{(2s+1)B_{2s+2}}{(2s+2)!},
\]
so to obtain the left hand identity in $(\ref{eq:R145})$ we only need to show that
\[
R_{2s+1}(2)=R_{2s+1}(1).
\]
We recall that the generating function for the Bernoulli numbers is
\[
\frac{t}{e^t-1}=\sum_{n=0}^\infty\frac{B_n t^n}{n!},
\]
and that for $j\geq 1$, $B_{2j+1}=0$. Hence
\[
\left (\frac{t}{e^t-1}\right )\left (\frac{t}{e^t-1}\right )=\frac{t^2}{(e^t-1)^2}
=\left (\sum_{k=0}^\infty\frac{B_k t^k}{k!}\right )\left (\sum_{n=0}^\infty\frac{B_n t^n}{n!}\right ),
\]
and for $s\geq 1$, the coefficient of $t^{2s+2}$ in this product of sums is
\begin{eqnarray*}
\sum_{k=0}^{2s+2}\frac{B_{2s+2-k} B_{k}}{(2s+2-k)!(k)!}1^{k}
&=&\sum_{k\, \mathrm{even}}^{2s+2}\frac{B_{2s+2-k} B_{k}}{(2s+2-k)!(k)!}1^{k}\\
=\sum_{k=0}^{s+1}\frac{B_{2s+2-2k} B_{2k}}{(2s+2-2k)!(2k)!}1^{2k}&=&R_{2s+1}(1).
\end{eqnarray*}
We notice also that
\[
\frac{\rm d}{{\rm d}t}\left (\frac{t^2}{e^t-1} \right )
=\frac{2t}{e^t-1}-\frac{t^2}{e^t-1}-\frac{t^2}{(e^t-1)^2}.
\]
Hence
\[
\frac{t^2}{(e^t-1)^2}=\frac{2t}{e^t-1}-\frac{t^2}{e^t-1}-\frac{\rm d}{{\rm d}t}\left (\frac{t^2}{e^t-1} \right )
\]
\[
\Rightarrow \frac{t^2}{(e^t-1)^2}=\sum_{n=0}^\infty\frac{B_n t^n}{n!}(2-t)
-\frac{\rm d}{{\rm d}t}\left ( \sum_{n=0}^\infty\frac{B_n t^{n+1}}{n!}\right )
\]
\[
=\sum_{n=0}^\infty\frac{B_n }{n!}\left (2t^n-t^{n+1}-(n+1)t^{n}\right )
=\sum_{n=0}^\infty\frac{B_n }{n!}\left (-(n-1)t^n-t^{n+1}\right ).
\]
So for $n=2s+2$ with $s\geq 1$, the coefficient of $t^{2s+2}$ in the above sum is given by
\[
-\frac{(n-1)B_n}{n!}=-\frac{(2s+1)B_{2s+2}}{(2s+2)!},
\]
and equating the two different expressions for the coefficients of $t^{2s+2}$ gives
\[
R_{2s+1}(1)=R_{2s+1}(2)=-\frac{(2s+1)B_{2s+2}}{(2s+2)!},
\]
where the right-hand identity in $(\ref{eq:R145})$ is obtained in a similar fashion. Hence $R_{2s}(1/2)=0$ and
applying $(\ref{eq:R14})$ we deduce the two expressions in $(\ref{eq:R15})$.

It was shown in  \cite{murty1} that when $s$ is even $R_{2s+1}(i)=0$. To prove the remaining identity in $(\ref{eq:R16})$ and also the second Corollary we need to introduce Grosswald's generalisation of Ramanujan's formula given in $(\ref{eq:R3})$.

Grosswald defines
\[
\sigma_t(n)=\sum_{d\mid n}d^t,\quad \hbox{\rm and}
\quad F_s(z)=\sum_{n=1}^{\infty}\sigma_{-s}(n)e^{2\pi i n z}
=\sum_{n=1}^{\infty}\frac{\sigma_{s}(n)}{n^s}e^{2\pi i n z},
\]
so that we may also write
\[
{F_s(z)=\sum_{n,\,m=1}^{\infty}n^{-s}e^{2\pi i nm z}}
=\sum_{n=1}^{\infty}\frac{1}{n^s}\left (\frac{e^{2\pi i n z}}{1-e^{2\pi i n z}}\right )
\]
\[
=-\zeta(s)-\sum_{n=1}^{\infty}\frac{1}{n^s(e^{2\pi i n z}-1)}=-\zeta(s)-F_s(-z).
\]
For any $z$ lying in the upper half-plane, Grosswald obtained
\be
F_{2s+1}(z)-z^{2s}F_{2s+1}\left (\frac{-1}{z} \right )
=\frac{1}{2}\zeta(2s+1)(z^{2s}-1)+\frac{(2\pi i)^{2s+1}}{2z}R_{2s+1}(z),
\label{eq:R24}
\ee
and he set $z=i\sqrt{\beta/\alpha}=i\beta/\pi$ in $(\ref{eq:R24})$ to get Ramanujan's formula $(\ref{eq:R3})$.

Substituting $(\ref{eq:R24})$ into $(\ref{eq:R14})$ we have
\[
\displaystyle{\frac{(2\pi i)^{2s+1}}{8z}\mathcal{R}_{2s}(z)=-\frac{1}{2}\zeta(2s+1)\left (\frac{z^{2s}}{2}-1+\frac{1}{2^{2s+1}}\right )}
\]
\be
+F_{2s+1}(z)-z^{2s}F_{2s+1}\left (\frac{-1}{z} \right )
-\frac{1}{2^{2s+1}}F_{2s+1}(2z)+\frac{z^{2s}}{2}F_{2s+1}\left (\frac{-1}{2z} \right ),
\label{eq:R25}
\ee
and setting $z=i/2$ in $(\ref{eq:R25})$ when $s$ is even gives
\be
\frac{(2\pi )^{2s+1}}{4}R_{2s}(i/2)=F_{2s+1}(i/2)-\frac{1}{2^{2s}}F_{2s+1}(2i)+\frac{1}{2}\eta(2s+1),
\label{eq:R26}
\ee
and when $s$ is odd
\be
-\frac{(2\pi )^{2s+1}}{4}R_{2s}(i/2)=F_{2s+1}(i/2)-\frac{1}{2^{2s}}F_{2s+1}(i)
+\frac{1}{2^{2s}}F_{2s+1}(2i)+\frac{1}{2}\zeta(2s+1).
\label{eq:R27}
\ee
Similarly for $z=i$ in $(\ref{eq:R25})$ when $s$ is even we have
\be
\frac{(2\pi )^{2s+1}}{4}R_{2s}(i)=F_{2s+1}(i/2)-\frac{1}{2^{2s}}F_{2s+1}(2i)+\frac{1}{2}\eta(2s+1),
\label{eq:R28}
\ee
and when $s$ is odd
\be
-\frac{(2\pi )^{2s+1}}{4}R_{2s}(i)=-F_{2s+1}(i/2)+4F_{2s+1}(i)-\frac{1}{2^{2s}}F_{2s+1}(2i)+\frac{1}{2}\left (3-\frac{1}{2^{2s}}\right )\zeta(2s+1).
\label{eq:R29}
\ee
When $s$ is even $(\ref{eq:R26})$ and $(\ref{eq:R28})$ together imply that $R_{2s}(i)=R_{2s}(i/2)$, which is the final
expression of the first Corollary.

It can be deduced from $(\ref{eq:R27})$ and $(\ref{eq:R29})$, using a similar approach to Murty et al. \cite{gun}, that $(\ref{eq:R22})$ is true for every integer $s\geq 1$. Applying the same method to $(\ref{eq:R26})$ (or $(\ref{eq:R28})$) in order that we may prove $(\ref{eq:R23})$, we argue as follows. Let $s$ be even. Then
\[
\frac{(2\pi )^{2s+1}}{4}R_{2s}(i/2)=F_{2s+1}(i/2)-\frac{1}{2^{2s}}F_{2s+1}(2i)+\frac{1}{2}\eta(2s+1)
\]
\[
=\sum_{n=1}^{\infty}\frac{1}{n^{2s+1}}\left (\frac{e^{-\pi  n }}{1-e^{-\pi  n }}\right )
-\frac{1}{2^{2s}}\sum_{n=1}^{\infty}\frac{1}{n^{2s+1}}\left (\frac{e^{-4\pi  n }}{1-e^{-4\pi  n }}\right )+\frac{1}{2}\eta(2s+1)
\]
\[
=\sum_{n=1}^{\infty}\frac{1}{n^{2s+1}}\left (\frac{1}{e^{\pi  n }-1}-\frac{1}{2^{2s}(e^{4\pi  n }-1)}\right )
+\frac{1}{2}\eta(2s+1).
\]
The right-hand side of this equation is is a sum of positive terms, and so is non-zero. The left-hand side is therefore a non-zero rational multiple of $\pi^{2s+1}$, where $s$ is an even integer. Consequently, for every integer $s\geq 1$, at least one of
\[
\zeta(4s+1),\qquad \sum_{n=1}^{\infty}\frac{1}{n^{4s+1}}\left (\frac{1}{e^{\pi  n }-1}-\frac{1}{2^{4s}(e^{4\pi  n }-1)}\right )
\]
is transcendental.

\end{proof}

\begin{proof}[Proof of Theorem 1.2]

\mbox{ }\\
\indent
From Lemma 2.1 (i) we have
\[
\left ( 1-\frac{1}{2^{2s}}\right )\frac{2B_{2s}}{s}=-\frac{1}{2s}\sum_{k=0}^{2s-2}\binom{2s}{k}2^{k-2s+1}B_k,
\]
so that
\[
\left ( 1-\frac{1}{2^{2s}}\right )B_{2s}=-\frac{s}{4s}\left ( \binom{2s}{1}2^{2-2s}B_1+\sum_{k=0}^{s-1}\binom{2s}{2k}2^{2k-2s+1}B_{2k}\right ).
\]
Hence
\[
\theta(2s)=\left ( 1-\frac{1}{2^{2s}}\right )\frac{(-1)^{s-1}2^{2s-1}\pi^{2s}B_{2s}}{(2s)!}
\]
\[
=\frac{(-1)^{s-1}\pi^{2s}}{4}\left ( \frac{2s}{(2s)!}-\sum_{k=0}^{s-1}\frac{2^{2k}B_{2k}}{(2s-2k)!(2k)!}\right )
\]
\[
=(-1)^{s-1}\pi^{2s}\left ( \frac{2s-1}{4(2s)!}-\sum_{k=1}^{s-1}\frac{2^{2k}B_{2k}}{4(2s-2k)!(2k)!}\right ),
\]
yielding
\[
\theta(2s)=(-1)^{s-1}\left ( \frac{(2s-1)\pi^{2s}}{4(2s)!}+\sum_{k=1}^{s-1}\frac{(-1)^k \pi^{2s-2k}}{2(2s-2k)!}\zeta(2k)\right ),
\]
which is the required identity given in $(\ref{eq:n4})$.

To obtain $(\ref{eq:n5})$ and $(\ref{eq:n10})$ we use the definition in $(\ref{eq:n16})$ so that
\[
\phi(2s)=\frac{(-1)^{s-1}\pi^{2s}B^\prime_{2s}}{2(2s)!}=\frac{(-1)^s \pi^{2s}}{(2s+1)!}\sum_{k=0}^{2s-1}\binom{2s+1}{k}2^{-1}B_k.
\]
\[
=\frac{(-1)^s \pi^{2s}}{(2s+1)!}\left ( \frac{(2s+1)B_1}{2}+\sum_{k=0}^{s-1}\binom{2s+1}{2k}\frac{B_{2k}}{2}\right )
\]
\[
=(-1)^{s-1} \pi^{2s}\left (
\frac{(2s+1)}{4(2s+1)!}-\frac{1}{2(2s+1)!}+\sum_{k=1}^{s-1}\frac{2^{2k-1}\pi^{2k}B_{2k}}{(2s-2k+1)!(2k)!2^{2k}\pi^{2k}}\right ),
\]
giving the required expressions
\[
\phi (2s)=(-1)^{s-1} \left ( \frac{(2s-1)\pi^{2s}}{4(2s+1)!}+\sum_{k=1}^{s-1}\frac{(-1)^k\pi^{2s-2k}}{(2s-2k+1)!2^{2k}}\zeta (2k)\right )
\]
\[
=(-1)^{s-1} \left ( \frac{(2s-1)\pi^{2s}}{4(2s+1)!}+\sum_{k=1}^{s-1}\frac{(-1)^k\pi^{2k}}{(2k+1)!2^{2s-2k}}\zeta (2s-2k)\right ),
\]
and
\[
\phi(2s)=(-1)^{s-1} \left ( \frac{(2s-1)\pi^{2s}}{4(2s+1)!}+\sum_{k=1}^{s-1}\frac{(-1)^{s-k}\pi^{2s-2k}}{(2s-2k+1)!}\phi (2k)\right )
\]
\[
=(-1)^{s-1} \left ( \frac{(2s-1)\pi^{2s}}{4(2s+1)!}+\sum_{k=1}^{s-1}\frac{(-1)^{s-k}\pi^{2k}}{(2k+1)!}\phi (2s-2k)\right ).
\]
We substitute $(\ref{eq:1})$ and $(\ref{eq:n5})$ into the first identity in $(\ref{eq:n3})$ to obtain
\[
\theta(2s)=\zeta(2s)-\phi(2s)
\]
\[
=(-1)^{s-1} \pi^{2s}\left ( \frac{(4s-2s+1)}{4(2s+1)!}+\sum_{k=1}^{s-1}\frac{(-1)^k}{(2s-2k+1)!\pi^{2k}}(\zeta(2k)-\phi (2k))\right )
\]
\[
=(-1)^{s-1} \pi^{2s}\left ( \frac{(4s-2s+1)}{4(2s+1)!}+\sum_{k=1}^{s-1}\frac{(-1)^k}{(2s-2k+1)!\pi^{2k}}\theta(2k)\right )
\]
\[
=(-1)^{s-1} \left ( \frac{\pi^{2s}}{4(2s)!}+\sum_{k=1}^{s-1}\frac{(-1)^{s-k}\pi^{2k}}{(2k+1)!}\theta(2s-2k)\right ),
\]
which is the expression in $(\ref{eq:n8})$.

Finally, to obtain $(\ref{eq:n11})$ we substitute $(\ref{eq:n8})$ and $(\ref{eq:n10})$ into the identity
\[
\eta(2s)=\theta(2s)-\phi(2s),
\]
given by the second relation in $(\ref{eq:n3})$.
\end{proof}

\section{Families of determinant equations}
This section describe some of the fundamental relationships between the three types of MCL determinant and certain recurrence relations.
\begin{lemma}
Let $h_1,h_2,\ldots,h_s$, $H_1,H_2,\ldots,H_s$ and $G_1,G_2,\ldots,G_s$ be given.
For $k=1,\ldots, s$, let $\Delta_k({\bf h})$ be the $k\times k$ type {\rm 1} MCL determinant in $(\ref{eq:d1})$. Let
$\Psi_k({\bf h},{\bf H})$ be the $k\times k$ type {\rm 2} MCL determinant in $(\ref{eq:d11})$ and let
$\Lambda_k({\bf h},{\bf H},{\bf G})$ be the $k\times k$ type {\rm 3} MCL determinant in $(\ref{eq:d12})$.
Let $\Delta_0({\bf h})=\Psi_0({\bf h},{\bf H})=\Lambda_0({\bf h},{\bf H},{\bf G})=~1$. Then
\begin{align}
\Delta_s({\bf h})&=-\sum_{k=0}^{s-1}h_{s-k}\Delta_k({\bf h}),
\label{eq:d2}\\
\Psi_s({\bf h},{\bf H})&=-\sum_{k=0}^{s-1}H_{s-k}\Delta_k({\bf h}),
\label{eq:d21}\\
\Psi_s({\bf h},{\bf H})&=-H_s-\sum_{k=1}^{s-1}h_{s-k}\Psi_k({\bf h},{\bf H}),
\label{eq:d22}\\
\Lambda_s({\bf h},{\bf H},{\bf G}) &=-\sum_{k=0}^{s-1}H_{s-k}\Psi_k({\bf h},{\bf G}).
\label{eq:d225}
\end{align}
Conversely, if $\Delta_0({\bf h})=\Psi_0({\bf h},{\bf H})=\Lambda_0({\bf h},{\bf H},{\bf G})=1$ and $\Delta_1({\bf h}),\ldots,\Delta_s({\bf h})$, $h_1,\ldots,h_s$ satisfy $(\ref{eq:d2})$, then
$\Delta_s({\bf h})$ is given in terms of $h_1,\ldots,h_s$ by the MCL determinant $(\ref{eq:d1})$. In addition, if $\Psi_1({\bf h},{\bf H}),\ldots,\Psi_s({\bf h},{\bf H})$ and $H_1,\ldots,H_s$
satisfy either of $(\ref{eq:d21})$ or $(\ref{eq:d22})$ (one implies the other), then $\Psi_s({\bf h},{\bf H})$ is given in terms of $h_1,\ldots,h_s$ and
$H_1,\ldots,H_s$ by the half-weighted MCL determinant $(\ref{eq:d11})$. As a further addition, if $\Lambda_1({\bf h},{\bf H},{\bf G}),\ldots,\Lambda_s({\bf h},{\bf H},{\bf G})$ and $G_1,\ldots,G_s$ also satisfy $(\ref{eq:d225})$ then $\Lambda_s({\bf h},{\bf H},{\bf G})$ is given in terms of $h_1,\ldots,h_s$, $H_1,\ldots,H_{s-1}$ and $G_1,\ldots,G_s$, by the fully-weighted MCL determinant in $(\ref{eq:d12})$.

We refer to the above recurrence relations according to which type of MCL determinant they relate to; \emph{type 1} recurrence relations are of the form $(\ref{eq:d2})$; \emph{type 2} recurrence relations are of the form
$(\ref{eq:d21})$ or $(\ref{eq:d22})$ and \emph{type 3} recurrence relations, where $\Psi_s$ satisfies a type {\rm 2} recurrence relation, are of the form $(\ref{eq:d225})$.
\end{lemma}
\begin{corollary}
Let $U_s$, $V_s$ and $W_s$ be the respective $s\times s$ matrices corresponding to the determinants $\Delta_s({\bf h})$, $\Psi_s({\bf h},{\bf H})$ and $\Lambda_s({\bf h},{\bf H},{\bf G})$, i.e.
\[
\Delta_s({\bf h})=|U_s|,\qquad \Psi_s({\bf h},{\bf H})=|V_s|,\qquad
\Lambda_s({\bf h},{\bf H},{\bf G})=|W_s|.
\]
Then denoting the characteristic polynomials of $U_s$, $V_s$ and $W_s$ by
\[
\Delta_s^{(\mu)}({\bf h})=|U_s-\mu I_s|,\qquad \Psi_s^{(\mu)}({\bf h},{\bf H})=|V_s-\mu I_s|,
\qquad \Lambda_s^{(\mu)}({\bf h},{\bf H},{\bf G})=|W_s-\mu I_s|,
\]
we find that the characteristic polynomials of $U_s$, $V_s$ and $W_s$ also satisfy the recurrence relations of the lemma, but with $h_1$ replaced with $h_1-\mu$, $H_1$ replaced with $H_1-\mu$, $G_1$ replaced with $G_1-\mu$ and $\Delta_s({\bf h})$, $\Psi_s({\bf h},{\bf H})$, $\Lambda_s({\bf h},{\bf H},{\bf G})$ respectively replaced by $\Delta_s^{(\mu)}({\bf h})$,
$\Psi_s^{(\mu)}({\bf h},{\bf H})$, $\Lambda_s^{(\mu)}({\bf h},{\bf H},{\bf G})$.
\end{corollary}
\begin{proof}
To obtain (\ref{eq:d2}), we expand the determinant in (\ref{eq:d1}) along its first column starting at the $r$-th row so that
\[
(-1)^s\Delta_s({\bf h})= (-1)^{s-1}1^{s-1}h_s(-1)^0\Delta_0({\bf h})+(-1)^{s-2}1^{s-2}h_{s-1}(-1)^1\Delta_1({\bf h})+
\]
\[
(-1)^{s-3}1^{s-3}h_{s-2}(-1)^2\Delta_2({\bf h})
+\ldots
+h_1(-1)^{s-1}\Delta_{s-1}({\bf h})=(-1)^{s-1}\sum_{k=0}^{s-1}h_{s-k}\Delta_k({\bf h})
\]
and hence the result. Similarly, for (\ref{eq:d21}), we expand the determinant in (\ref{eq:d11}) along its first column starting at the $r$-th row,
yielding
\[
(-1)^s\Psi_s({\bf h},{\bf H})= (-1)^{s-1}1^{s-1}H_s(-1)^0\Delta_0({\bf h})+(-1)^{s-2}1^{s-2}H_{s-1}(-1)^1\Delta_1({\bf h})+
\]
\[
(-1)^{s-3}1^{s-3}H_{s-2}(-1)^2\Delta_2({\bf h})
+\ldots
+H_1(-1)^{s-1}\Delta_{s-1}({\bf h})=(-1)^{s-1}\sum_{k=0}^{s-1}H_{s-k}\Delta_k({\bf h}).
\]
To obtain (\ref{eq:d22}), we expand the determinant in (\ref{eq:d11}) along its $r$-th row starting at the $1$-st column, giving
\[
(-1)^s\Psi_s({\bf h},{\bf H})= (-1)^{s-1}1^{s-1}H_s+(-1)^{s-2}1^{s-2}h_{s-1}(-1)^1\Psi_1({\bf h},{\bf H})+
\]
\[
(-1)^{s-3}1^{s-3}h_{s-2}(-1)^2\Psi_2({\bf h},{\bf H})
+\ldots
+h_1(-1)^{s-1}\Psi_{s-1}({\bf h},{\bf H})
\]
\[
=(-1)^{s-1}\left (H_s+ \sum_{k=0}^{s-1}h_{s-k}\Psi_k({\bf h},{\bf H})\right ).
\]
Finally, to deduce (\ref{eq:d225}), we expand the determinant in (\ref{eq:d12}) along its $1$-st column starting at the $r$-th row, which yields
\[
(-1)^s\Lambda_s({\bf h},{\bf H},{\bf G})= (-1)^{s-1}1^{s-1}H_s+(-1)^{s-2}1^{s-2}H_{s-1}\left |G_1\right |+
\]
\[
(-1)^{s-3}1^{s-3}H_{s-2}\left |
\begin{array}{cc}
h_1 & 1\\
G_2 & G_1\\
\end{array}
\right |
+\ldots
+H_1\left |
\begin{array}{cccc}
h_1 & 1 &   \ldots & 0 \\
h_2 & h_1 &  \ldots & 0 \\
\vdots & \vdots &  \ddots & \vdots\\
h_{s-2} & h_{s-1} &  \ldots & 1\\
G_{s-1} & G_{s-2} & \ldots & G_1\\
\end{array}
\right |,
\]
and by comparing the above determinants with those of the form $(\ref{eq:d11})$ we obtain the result. The converse follows by showing inductively that each $\Delta_s({\bf h})$, $\Psi_s({\bf h},{\bf H})$ and $\Lambda_s({\bf h},{\bf H},{\bf G})$ can be expressed as a determinant of the required form and then re-packing the original determinants expanded above.

To see the Corollary one simply replaces $h_1$, $H_1$ and $G_1$ with $h_1-\mu$, $H_1-\mu$ and $G_1-\mu$ respectively in the recurrence relations of the Lemma.
\end{proof}

\begin{rmk}
We note that the symmetric structure of the $n\times n$ MCL determinant $\Delta_s({\bf h})=|U_s|$, in $(\ref{eq:d1})$, leads to a symmetry in the cofactors of the matrix $U=(u_{i,j})$. Specifically, let $M_{i,j}$ be the cofactor or minor of $u_{i,j}$. Then for $i-j\geq 0$ we have
\be
M_{i,j}=(-1)^{i+j}\Delta_{n-i}({\bf h})\times \Delta_{j-1}({\bf h}).
\label{eq:d226}
\ee
\end{rmk}

\begin{lemma}
With $\Delta_s({\bf h})$ as defined in the previous lemma we have
\be
\Delta_s({\bf h})=
\mathop{\mathop{\sum_{t=1}^s \sum_{d_i\geq 0}}_{d_1+d_2+\ldots +d_s=t}}_{d_1+2d_2+\ldots+sd_s=s} (-1)^{t}\binom{t}{d_1,d_2,\ldots,d_s}
h_1^{d_1}h_2^{d_2}\ldots h_s^{d_s},
\label{eq:d23}
\ee
where the above sum consists of $2^{s-1}$ monomial terms.
\end{lemma}
\begin{proof}
Repeated use of $(\ref{eq:d2})$ gives
\[
\Delta_s({\bf h})=(-1)^w\sum_{k_1=0}^{s-1}\sum_{k_2=0}^{k_1-1}\ldots \sum_{k_w=0}^{k_{w-1}-1}
h_{s-k_1}h_{k_1-k_2}\ldots h_{k_{w-1}-k_{w}}\Delta_{k_w}({\bf h}),
\]
with $k_w=k_{w-1}-1=0$, so that $\Delta_{k_w}({\bf h})=\Delta_0({\bf h})=1$. Hence we can write
\be
\Delta_s({\bf h})=(-1)^w\sum_{k_1=0}^{s-1}\sum_{k_2=0}^{k_1-1}\ldots \sum_{k_w=0}^{k_{w-1}-1}
h_{s-k_1}h_{k_1-k_2}\ldots h_{k_{w-1}-k_{w}},
\label{eq:d24}
\ee
which is just a sum of products of $h_k$, where the subscripts in each product sum to $s$.
Therefore we have established that $\Delta_s({\bf h})$ is a sum of monomials of the form
\be
\pm h_{1}^{d_1}h_{2}^{d_2}\ldots h_{s}^{d_s},
\label{eq:d25}
\ee
with
\[
d_i\geq 0,\qquad d_1+2 d_2+\ldots+s d_s=s.
\]
We note that for a given $d_1+2 d_2+\ldots+s d_s=s$, with $d_1+d_2+\ldots +d_s=t$, the coefficient of the product
in $(\ref{eq:d25})$ is the same (ignoring sign) as that in the multinomial expansion of
\[
(h_{1}+h_{2}+\ldots +h_{s})^t.
\]
Hence we can write
\be
\Delta_s({\bf h})= \mathop{\mathop{\sum_{t=1}^s \sum_{d_i\geq 0}}_{d_1+d_2+\ldots +d_s=t}}_{d_1+2d_2+\ldots+sd_s=s}(-1)^t \binom{t}{d_1,d_2,\ldots,d_s}
h_{1}^{d_1}h_{2}^{d_2}\ldots h_{s}^{d_s}.
\label{eq:d26}
\ee
To see that the number of monomials in this expression for $\Delta_s({\bf h})$ is equal to $2^{s-1}$, we note that for fixed values of $t$, the rules of
summation give the number of compositions of the integer $s$ into $t$ parts, which is
known to be $\binom{s-1}{t-1}$. Explicitly we have
\be
\mathop{\mathop{\sum_{d_i\geq 0}}_{d_1+d_2+\ldots +d_s=t}}_{d_1+2d_2+\ldots+sd_s=s} \binom{t}{d_1,d_2,\ldots,d_s}
=\binom{s-1}{t-1},
\label{eq:d27}
\ee
so that summing over the values $1\leq t \leq s$ gives the required total of $2^{s-1}$ monomials.
For example, when $s=5$, we have
\[
\Delta_5({\bf h})= -h_1^5 + 4h_1^3 h_2 -(3 h_1 h_2^2 + 3h_1^2 h_3) + (2h_2 h_3 +2 h_1 h_4) -h_5.
\]
\end{proof}
\begin{defn}
Let the five infinite dimensional vectors ${\bf u}$, ${\bf v}$, ${\bf U}_1$, ${\bf U}_2$ and ${\bf U}_3$ be defined such that
\[
{\bf u}\phantom{1}=\left (\frac{1}{3!},\,\frac{1}{5!},\,\frac{1}{7!},\,\ldots,\frac{1}{(2s+1)!},\,\ldots \right ),\qquad
{\bf v}\phantom{1}=\left (\frac{1}{2!},\,\frac{1}{3!},\,\frac{1}{4!},\,\ldots,\frac{1}{(s+1)!},\,\ldots \right )
\]
\[
{\bf U}_1=\left (\frac{1}{3!},\,\frac{2}{5!},\,\frac{3}{7!},\,\ldots,\frac{s}{(2s+1)!},\,\ldots \right ),\qquad
{\bf U}_2=\left (\frac{1}{3!},\,\frac{3}{5!},\,\frac{5}{7!},\,\ldots,\frac{2s-1}{(2s+1)!},\,\ldots \right )
\]
\be
{\bf U}_3=\left (\frac{1}{2!},\,\frac{1}{4!},\,\frac{1}{6!},\,\ldots,\,\frac{1}{(2s)!},\,\ldots \right )\\
\ee
\end{defn}
We now give two well known MCL determinant identities for the Bernoulli numbers \cite{woon}. In the first instance we have
\be
B(s)=s!\Delta_s({\bf v}),
\label{eq:d31}
\ee
which by Lemma 3.1 can be written as the recurrence relation
\[
B_s=-\frac{1}{s+1}\sum_{k=0}^{s-1}\binom{s+1}{k}B_k,
\]
given in $(\ref{eq:n01})$, and by Lemma 3.2, as the double sum
\be
\frac{B(s)}{s!}=
\mathop{\mathop{\sum_{t=1}^s \sum_{d_i\geq 0}}_{d_1+d_2+\ldots +d_s=t}}_{d_1+2d_2+\ldots+sd_s=s} \binom{t}{d_1,d_2,\ldots,d_s}
\frac{(-1)^{t}}{2!^{d_1}3!^{d_2}\ldots (s+1)!^{d_s}}.
\label{eq:d32}
\ee
The second identity states that
\be
B(2s)=-\frac{-(2s)!}{2(2^{2s-1}-1)}\Delta_s({\bf u}),
\label{eq:d325}
\ee
the equivalent forms of which we express in terms of $\eta(2s)$ in Lemma 3.3.

Lemmas 3.1 and 3.2 effectively give us two alternative ways to express an MCL determinant; as a recurrence relation and as a double sum over compositions
of $s$ into $t$ parts. For a half-weighted MCL determinant we get the two alternatives just stated, along with an extra recurrence relation obtained by expanding the determinant along the $r$-th row instead of the first column.

Combining these equivalent methods of expression for $\eta(2s)$, $\zeta(2s)$, $\theta(2s)$ and $\phi(2s)$ we obtain the following lemma.

\begin{lemma}
Let $\eta(2s)$, $\zeta(2s)$, $\theta(2s)$ and $\phi(2s)$ be defined as in Theorem~{ 1.1}. Then
\begin{eqnarray*}
2\eta(2s)&=&
(-1)^s\pi^{2s}\Delta_s({\bf u})\\
&=& \frac{(-1)^{s-1}\pi^{2s}}{(2s+1)!}+\sum_{k=1}^{s-1}\frac{(-1)^{k-1}\pi^{2k}}{(2k+1)!}2\eta{(2s-2k)}\\
&=& \pi^{2s}\mathop{\mathop{\sum_{t=1}^s \sum_{d_i\geq 0}}_{d_1+d_2+\ldots +d_s=t}}_{d_1+2d_2+\ldots+sd_s=s} \binom{s}{d_1,d_2,\ldots,d_s}
\frac{(-1)^{t+s}}{3!^{d_1}5!^{d_2}\ldots (2s+1)!^{d_s}}.
\label{eq:d33}
\end{eqnarray*}
\begin{eqnarray*}
\zeta(2s)&=& (-1)^s\pi^{2s}\Psi_s({\bf u,\,\bf U}_1)\\
&=& \frac{(-1)^{s-1}s\pi^{2s}}{(2s+1)!}+\sum_{k=1}^{s-1}\frac{(-1)^{k-1}k\pi^{2k}}{(2k+1)!}2\eta{(2s-2k)}\\
 &=& \frac{(-1)^{s-1}s\pi^{2s}}{(2s+1)!}+\sum_{k=1}^{s-1}\frac{(-1)^{k-1}\pi^{2k}}{(2k+1)!}\zeta{(2s-2k)}\\
&=&\frac{2^{2s-2}\pi^{2s}}{(2^{2s-1}-1)}
\mathop{\mathop{\sum_{t=1}^s \sum_{d_i\geq 0}}_{d_1+d_2+\ldots +d_s=t}}_{d_1+2d_2+\ldots+sd_s=s} \binom{t}{d_1,d_2,\ldots,d_s}
\frac{(-1)^{t+s}}{3!^{d_1}5!^{d_2}\ldots (2s+1)!^{d_s}}.
\label{eq:d34}
\end{eqnarray*}
\begin{eqnarray*}
4\phi(2s)&=&(-1)^s\pi^{2s}\Psi_s({\bf u,\,\bf U}_2)\\
&=& \frac{(-1)^{s-1}(2s-1)\pi^{2s}}{(2s+1)!}+\sum_{k=1}^{s-1}\frac{(-1)^{k-1}(2k-1)\pi^{2k}}{(2k+1)!}2\eta{(2s-2k)}\\
&=& \frac{(-1)^{s-1}(2s-1)\pi^{2s}}{(2s+1)!}+\sum_{k=1}^{s-1}\frac{(-1)^{k-1}\pi^{2k}}{(2k+1)!}4\phi{(2s-2k)}\\
&=&\frac{\pi^{2s}}{(2^{2s-1}-1)}\mathop{\mathop{\sum_{t=1}^s \sum_{d_i\geq 0}}_{d_1+d_2+\ldots +d_s=t}}_{d_1+2d_2+\ldots+sd_s=s}
\binom{s}{d_1,d_2,\ldots,d_s} \frac{(-1)^{t+s}}{3!^{d_1}5!^{d_2}\ldots (2s+1)!^{d_s}}.
\label{eq:d36}
\end{eqnarray*}
\begin{eqnarray*}
4\theta(2s)&=& (-1)^s\pi^{2s}\Psi_s({\bf u,\,\bf U}_3)\\
&=& \frac{(-1)^{s-1}\pi^{2s}}{(2s)!}+\sum_{k=1}^{s-1}\frac{(-1)^{k-1}\pi^{2k}}{(2k)!}2\eta{(2s-2k)}\\
&=& \frac{(-1)^{s-1}\pi^{2s}}{(2s)!}+\sum_{k=1}^{s-1}\frac{(-1)^{k-1}\pi^{2k}}{(2k+1)!}4\theta{(2s-2k)}\\
&=&\frac{(2^{2s}-1)\pi^{2s}}{(2^{2s-1}-1)}\mathop{\mathop{\sum_{t=1}^s \sum_{d_i\geq 0}}_{d_1+d_2+\ldots +d_s=t}}_{d_1+2d_2+\ldots+sd_s=s}
\binom{s}{d_1,d_2,\ldots,d_s} \frac{(-1)^{t+s}}{3!^{d_1}5!^{d_2}\ldots (2s+1)!^{d_s}}.
\label{eq:d35}
\end{eqnarray*}
\end{lemma}
\begin{proof}
The proofs follow directly by applying Lemmas 3.1 and 3.2 to the recurrence relations stated in Theorem 1.2.
\end{proof}
We are now in a position to prove Theorem 1.3 and Theorem 1.4.
\begin{proof}[Proof of Theorem 1.3]

\mbox{ }\\
\indent
From Lemma $2.1$, we can write
\[
8\sum_{s=1}^\infty (-1)^{s-1}\frac{\theta(2s)}{\pi^{2s}}x^{2s-2}\left ( 1+ 2\sum_{s=1}^\infty (-1)^{s-1}\frac{\phi(2s)}{\pi^{2s}}x^{2s}\right )^{-1}=
\left ( 1+ 2S\right )^{-1}
\]
\[
=1-2S+(2S)^2-(2S)^3+(2S)^4-\ldots,
\]
and comparing terms in $x$ enables us to obtain the double-sum expression for
$\theta(2s)$ in $(\ref{eq:n36})$. The determinant and single-sum identities of $\theta(2s)$ in $(\ref{eq:n36})$ are then deduced by applying Lemmas 3.1
and 3.2. To obtain the expressions in $(\ref{eq:n37})$ we rearrange the double-sum in $(\ref{eq:n36})$ and again apply Lemmas 3.1 and 3.2 to get the
determinant and single-sum identities.
\end{proof}

\begin{proof}[Proof of Theorem 1.4]

\mbox{ }\\
\indent
Substituting $\Psi_n=\lambda_n$, $H_n=-na_n$ and $h_{n-k}=a_{n-k}$ into the type 2 recurrence relation $(\ref{eq:d22})$ of Lemma 3.1 gives
\[
\lambda_n=n a_n-\sum_{j=1}^{n-1}\lambda_j a_{n-j},
\]
which is just $(\ref{eq:li2})$. Hence by Lemma 3.1 we have
\[
\lambda_n=M_n=(-1)^{n-1}\left |
\begin{array}{cccccc}
a_1 & 1 & 0 & 0 &  \ldots & 0 \\
2a_2 & a_1 & 1 & 0 &  \ldots & 0 \\
3a_3 & a_2 & a_1 & 1 &  \ldots & 0\\
\vdots & \vdots & \vdots & \vdots &  \ddots & \vdots\\
(n-1)a_{n-1} & a_{n-2} & a_{n-3} & a_{n-4} & \ldots & 1\\
na_n & a_{n-1} & a_{n-2} & a_{n-3} & \ldots & a_1\\
\end{array}
\right |=-\Psi_n({\bf a,\,\bf A}),
\]
so that
\[
\lambda_n \geq 0 \Leftrightarrow M_n\geq 0.
\]Therefore a necessary and sufficient condition for $\lambda_n$ to be non-negative for $n=1,2,3,\ldots$ is
for $M_n$ to be non-negative and the equivalence of the theorem follows.
\end{proof}

\begin{remark}

An important point to note with the sums in $(\ref{eq:n36})$ and $(\ref{eq:n37})$ is that they consist entirely of positive terms, whereas the single sums in Theorem 1.2 all consist of an alternating series. We now give the examples using the double-sum composition expressions from both $(\ref{eq:n37})$ and
$(\ref{eq:d34})$ and also the single recurrence sum from $(\ref{eq:n37})$ for $\zeta(14)=2\pi^{14}/18243225$.

Taking $s=6$ in $(\ref{eq:n37})$, we have the strictly positive sums
\begin{align*}
\zeta(14)=\frac{\pi^2}{2^{14}-1} \left( \zeta(12)+2 \left (2\zeta(2)\zeta(10)+2\zeta(4)\zeta(8)+2\zeta^2(6)
\right )\right.\\
+2^2  \left( 3\zeta^2(2)\zeta(8)+6\zeta(2)\zeta(4)\zeta(6)+\zeta(4)^3\right )\\
+2^3\left (4\zeta^3(2)\zeta(6) +6\zeta^2(2)\zeta^2(4)\right )\\+2^45\zeta^4(2)\zeta(4)+2^5\zeta^6(2)\left. \right ),
\intertext{and}
\zeta(14)=4098\zeta(2)\zeta(12)+1038\zeta(4)\zeta(10)+318\zeta(6)\zeta(8),
\intertext{whereas $(\ref{eq:d34})$ with $s=7$ gives the alternating expansion}
\zeta(14)=\frac{\pi^{14}2^{12}}{2^{13}-1}
\left ( \frac{1}{15!}-\left (\frac{2}{3!13!}+\frac{2}{5!11!}+\frac{2}{7!9!}\right ) \right.\\
+\left (\frac{3}{3!^2 11!}+\frac{6}{3!5!9!}+\frac{3}{3!7!^2}+\frac{3}{5!^27!}\right )\\
-\left (\frac{4}{3!^3 9!}+\frac{12}{3!^2 5!7!}+\frac{4}{3!5!^3}\right )\\
+\left (\frac{5}{3!^4 7!}+\frac{10}{3!^3 5!^2}\right )\\
\left.-\frac{6}{3!^5 5!}+\frac{1}{3!^7}\right).
\end{align*}
As expected, the number of terms in the double sums from $(\ref{eq:n37})$ and $(\ref{eq:d34})$ is given by
\[
{{}^{s-1}\text{C}}_{t-1}2^{t-1}={{}^{5}\text{C}}_{t-1}2^{t-1},\qquad 1\leq t \leq 6,
\]
and
\[
{{}^{s-1}\text{C}}_{t-1}={{}^{6}\text{C}}_{t-1},\qquad 1\leq t \leq 7,
\]
respectively. In general, the sum of the coefficients in
$(\ref{eq:d34})$ is simply $2^{s-1}$, and for $(\ref{eq:n37})$, taking into account the extra $2^{t-1}$ terms, the coefficients sum to $3^{s-1}$.
\end{remark}

\noindent
For comparison we give the more common (alternating) recurrence identity \cite{borwein}
for $\zeta(2s)$, which states that
\be
\zeta(2s)=\sum_{k=1}^s\frac{(-1)^k\pi^{2k}}{(2k+1)!}(1-2^{2k-2s+1})\zeta(2s-2k)=0.
\label{eq:m402}
\ee

\section{Approximating equations}
The Riesz, Hardy-Littlewood and B\'aez-Duarte equivalences to the Riemann hypothesis all rely on bounding sums involving inverse zeta constants. Hence the
ability to approximate $1/\zeta(s)$, for $s\in\mathbb{N}$, is of interest. The simplest well-known bound for
$\zeta(s)$ and its inverse can be deduced as follows.

From Euler's product \cite{ayoub}, \cite{sondow}, and the identity
\[
2\theta(s)-1=2\left (1-\frac{1}{2^s}\right )\zeta(s)-1<\zeta(s),
\]
we obtain
\be
1 + \frac{1}{2^{s}-1}<\zeta(s)<1 + \frac{1}{2^{s-1}-1},
\label{eq:27}
\ee
and
\be
1-\frac{1}{2^{s-1}}<\frac{1}{\zeta(s)}<1-\frac{1}{2^{s}}.
\label{eq:28}
\ee
In a similar vein, using $(\ref{eq:n1})$ and $(\ref{eq:m1})$ gives
\be
1-\frac{1}{2^{s}-1}<\eta(s)<1<\theta(s)<1+\frac{1}{2^{s}-2},
\label{eq:29}
\ee
from which we obtain the consecutive integer bounds
\be
2^s-2<(2^s-1)\eta(s)<2^s-1,\qquad
2^s-2<(2^s-2)\theta(s)<2^s-1.
\label{eq:31}
\ee
For $\zeta(s)$ and $\phi(s)$ we have
\be
2^s-2\zeta(s)<(2^s-3)\zeta(s)<2^s-\zeta(s),\qquad
1-2\phi(s)<(2^s-3)\phi(s)<1-\phi(s),
\label{eq:325}
\ee
so that the respective intervals here are $\zeta(s)$ and $\phi(s)$ themselves.

Although the upper bound in $(\ref{eq:27})$ has been improved by Murty et al \cite{murty1} to
$1+2^{-s}(s+1)/(s-1)$, the interval bounding $\zeta(s)$ is still $O\left (1/2^{s-1}\right )$. The bounds
in Theorem 1.5 for the pseudo characteristic polynomials that approximate $\zeta(2s)$ and $\zeta(2s-1)$
are more accurate. We need two lemmas before we prove Theorem~1.5.

\begin{lemma}
Let
\[
F_0(s)=\frac{{\pi}^{s}}{s!},\,\,F_1(s)=\frac{F_0(s)}{\{\zeta(s)\}^2},\,\,F_2(s)=\frac{F_0(s)}{\{\theta(s)\}^2},\,\,
F_3(s)=\frac{F_0(s)}{\{\zeta(s)\}^3},\,\,F_4(s)=\frac{F_0(s)}{\{\theta(s)\}^3},
\]
and $(t_0,t_1,t_2,t_3,t_4)=(9,34,76,68,228)$.
Then for integers $k\geq 1$ and $s\geq t_i$ we have
\be
F_i(s)\leq \frac{1}{(2k)^{s-t_i}},\qquad i=1,\ldots 4.
\label{eq:lem41}
\ee
\end{lemma}
\begin{corollary}
As $s\rightarrow \infty$ we have
\[
F(s)=o\left(\{\zeta(s)\}^3\right ),\quad\hbox{\rm and }\quad F(s)=o\left(\{\theta(s)\}^3\right ).
\]
\end{corollary}
\begin{proof}
For $i=0$ and $s\geq 9$ we can write $s=9k+r$ with $k\geq 1$, $r\geq 0$. Then $4^{9k}/(9k)!<1$ giving
\[
\frac{{\pi}^{s}}{s!}<\frac{{4}^{s}}{s!}\leq\frac{{4}^{9k+r}}{(9k)!(9k+1)\ldots(9k+r)}
\]
\[
\leq\frac{{4}^{r}}{(9k+1)\ldots(9k+r)}\leq \frac{{4}^{r}}{{(9k)}^{r}}\leq \frac{1}{(2k)^{r}}
\leq \frac{1}{(2k)^{s-9k}},
\]
and taking $k=1$ gives the result.

Similarly, for $i=1$ and $s\geq 34$ we can write $s=34k+r$ with $k\geq 1$, $r>0$. Then $13^{34k}/(34k)! <1$
and we have
\[
\frac{F_1(s)}{\{\zeta(s)\}^2}<\frac{{(4\pi)}^{s}}{s!}<\frac{{13}^{s}}{s!}
\leq\frac{{13}^{34k+r}}{(34k)!(34k+1)\ldots(34k+r)}
\]
\[
\leq\frac{{13}^{r}}{(34k+1)\ldots(34k+r)}\leq \frac{{13}^{r}}{{(34k)}^{r}}
\leq \frac{1}{(2k)^{r}}\leq \frac{1}{(2k)^{s-34k}}\leq \frac{1}{(2k)^{s-34}},
\]
when $k=1$. The proofs are similar for the remaining three cases when $i=2,3,4$.

The Corollary follows by considering the limit as $s\rightarrow \infty$ by either fixing $k$ and increasing $r$ or vice-versa.
\end{proof}
\begin{lemma}[approximate sine lemma] In the notation of $(\ref{eq:33})$ we have
\[
p_{s}(x)=1-\frac{(-1)^{[x]}\sin{\pi \{x\}}}{\pi x}+\sum_{k=s}^{\infty}\frac{(-1)^{k-1}(\pi x)^{2k}}{(2k+1)!},
\]
and
\[
q_{s}(x)=\frac{(-1)^{[x]}\sin{\pi \{x\}}}{\pi x}+\sum_{k=s}^{\infty}\frac{(-1)^{k}(\pi x)^{2k}}{(2k+1)!}.
\]
\end{lemma}
\begin{proof} We have
\[
p_{s}(x)=\frac{1}{\pi x}\sum_{k=1}^{s-1}\frac{(-1)^{k-1}(\pi x)^{2k+1}}{(2k+1)!}
=1-\frac{1}{\pi x}\sum_{k=0}^{s-1}\frac{(-1)^{k}(\pi x)^{2k+1}}{(2k+1)!}
\]
\[
=1-\frac{1}{\pi x}\sum_{k=0}^{\infty}\frac{(-1)^{k}(\pi x)^{2k+1}}{(2k+1)!}+\sum_{k=s}^{\infty}\frac{(-1)^{k-1}(\pi x)^{2k}}{(2k+1)!}
\]
\be
=1-\frac{\sin{\pi x}}{\pi x}+\sum_{k=s}^{\infty}\frac{(-1)^{k-1}(\pi x)^{2k}}{(2k+1)!}.
\label{eq:336}
\ee
\[
=1-\frac{\sin{\pi ([x]+\{x\})}}{\pi x}+\sum_{k=s}^{\infty}\frac{(-1)^{k-1}(\pi x)^{2k}}{(2k+1)!}.
\]
\[
=1-\frac{\sin{\pi [x]}\cos{\pi \{x\}}+\cos{\pi [x]}\sin{\pi \{x\}}}{\pi x}+\sum_{k=s}^{\infty}\frac{(-1)^{k-1}(\pi x)^{2k}}{(2k+1)!}.
\]
\[
=1-\frac{\cos{\pi [x]}\sin{\pi \{x\}}}{\pi x}+\sum_{k=s}^{\infty}\frac{(-1)^{k-1}(\pi x)^{2k}}{(2k+1)!}.
\]
\[
=1-\frac{(-1)^{[x]}\sin{\pi \{x\}}}{\pi x}+\sum_{k=s}^{\infty}\frac{(-1)^{k-1}(\pi x)^{2k}}{(2k+1)!},
\]
as required.
\end{proof}
We are now in a position to prove Theorem $1.5$. As with the previous lemma we give the proof for
$z_{s}(\zeta(k))$, with $k=2s$ or $k=2s-1$. The proofs for $t_s(\theta(k))$, $e_s(\eta(k))$, $f_s(\phi(k))$, $1+q_s(\theta(k))$ and $1+q_s(\eta(k))$ are similar.

\begin{proof}[Proof of Theorem 1.5]

\mbox{ }\\
\indent
By
$(\ref{eq:33a})$ and Lemma 4.2, for any integer $k\geq 1$, we can write
\[
z_{s}(\zeta(k))=\frac{(-1)^{s-1}\pi^{2s}s}{(2s+1)!}+p_{s}(\zeta(k))
\]
\[
= \frac{(-1)^{s-1}\pi^{2s}s}{(2s+1)!}+
1-\frac{(-1)^{1}\sin{\pi \{\zeta(k)\}}}{\pi \zeta(k)}+\sum_{k=s}^{\infty}\frac{(-1)^{k-1}(\pi \zeta(k))^{2k}}{(2k+1)!}
\]
\[
= \frac{(-1)^{s-1}\pi^{2s}}{(2s+1)!}\left (s+\zeta(k)^{2s}-\frac{\pi^2\zeta(k)^{2s+2}}{(2s+2)(2s+3)}
+\frac{\pi^4\zeta(k)^{2s+4}}{(2s+2)\ldots(2s+5)}-\ldots
\right )
\]
\[
+ 1+\frac{\sin{\pi \{\zeta(k)\}}}{{\pi \zeta(k)}}
\]
\[
\leq \left |\frac{(-1)^{s-1}\pi^{2s}}{(2s+1)!}\right |
\left (s+\zeta(k)^{2s}\left (1+\frac{(\pi\zeta(k))^{2}}{(2s+2)(2s+3)}
+\frac{(\pi\zeta(k))^{4}}{(2s+2)\ldots(2s+5)}+\ldots\right )
\right )
\]
\[
+ 1+\frac{\{\zeta(k)\}}{{\zeta(k)}}-\frac{\pi^2\{\zeta(k)\}^3}{{\zeta(k)}}+\ldots
\]
\[
\leq \frac{\pi^{2s}}{(2s+1)!}
\left (s+\zeta(k)^{2s}\left (1+\frac{(\pi\zeta(k))^{2}}{(2s+2)^2}
+\frac{(\pi\zeta(k))^{4}}{(2s+2)^4}+\ldots\right )
\right )
\]
\[
+ 1+\frac{\{\zeta(k)\}}{{\zeta(k)}}-\frac{\pi^2\{\zeta(k)\}^3}{{\zeta(k)}}+\ldots
\]
\[
\leq \frac{\pi^{2s}}{(2s+1)!}
\left (s+\zeta(k)^{2s}\left (1-\frac{(\pi\zeta(k))^{2}}{(2s+2)^2}
\right )^{-1}
\right )
+ 1+\frac{\{\zeta(k)\}}{{\zeta(k)}}-\frac{\pi^2\{\zeta(k)\}^3}{{\zeta(k)}}+\ldots
\]
\be
\leq \frac{\pi^{2s}}{(2s)!}
+ 1+\frac{\{\zeta(k)\}}{{\zeta(k)}}-\frac{\pi^2\{\zeta(k)\}^3}{{\zeta(k)}}+
\frac{\pi^4\{\zeta(k)\}^5}{{\zeta(k)}}-\ldots,\qquad \hbox{ for } k\geq 4
\label{eq:40}
\ee
\[
\leq \frac{\pi^{2s}}{(2s)!}
+ 1+\left (\{\zeta(k)\}-\pi^2\{\zeta(k)\}^3+
\pi^4\{\zeta(k)\}^5-\ldots\right )\left (1+\{\zeta(k)\}\right )^{-1}.
\]
Expanding out the brackets and collecting terms we find that
\[
z_{s}(\zeta(k))\leq 1+\{\zeta(k)\}+\frac{\pi^{2s}}{(2s)!}-\{\zeta(k)\}^2
+(1-\pi^2)\{\zeta(k)\}^3+O\left (\{\zeta(k)\}^4\right ).
\]
\[
\leq \zeta(k)+\frac{\pi^{2s}}{(2s)!}-\{\zeta(k)\}^2,
\]
and with a slight adjustment of signs to the above argument we can deduce the lower bound
\[
z_{s}(\zeta(k))\geq \zeta(k)-\frac{\pi^{k}}{(2s)!}-2\{\zeta(k)\}^2.
\]
Hence we have
\[
\zeta(k)-\frac{\pi^{2s}}{(2s)!}-2\{\zeta(k)\}^2\leq z_{s}(\zeta(k))\leq
\zeta(k)+\frac{\pi^{2s}}{(2s)!}-\{\zeta(k)\}^2,
\]
and applying Lemma 4.1 with $k=2s$ or $2s-1$ we deduce $(\ref{eq:65})$ of Theorem 1.5.

To see $(\ref{eq:67})$ we use $q_{s}(\zeta(k))=1-p_{s}(\zeta(k))$
in the above proof, omitting the initial term in $z_{s}(\zeta(k))$. This yields
\[
q_{s}(\zeta(k))\leq \frac{2(\pi^{2s})}{(2s+1)!}
-\frac{\{\zeta(k)\}}{{\zeta(k)}}+\frac{\pi^2\{\zeta(k)\}^3}{{\zeta(k)}}-
\frac{\pi^4\{\zeta(k)\}^5}{{\zeta(k)}}+\ldots,\qquad \hbox{ for } k\geq 4,
\]
so that
\[
q_{s}(\zeta(k))\leq \frac{1}{\zeta(k)}-1+\frac{2(\pi^{2s})}{(2s+1)!}
+\pi^2\{\zeta(k)\}^3.
\]
We again obtain a lower bound by considering the signs in the upper bound argument, and combining these results we have
\[
\frac{1}{\zeta(k)}-1-\frac{2(\pi^{2s})}{(2s+1)!}\leq q_{s}(\zeta(k))\leq \frac{1}{\zeta(k)}-1+\frac{2(\pi^{2s})}{(2s+1)!}
+\pi^2\{\zeta(k)\}^3,
\]
whence we apply Lemma 4.1 with $k=2s$ or $2s-1$ to deduce the inequality $(\ref{eq:67})$. The proofs for the inequalities involving $\theta(k)$ and $1/\theta(k)$ in $(\ref{eq:66})$ and
$(\ref{eq:68})$ are similar.
\end{proof}
Hence $1+q_s(\zeta(2s))$ approximates $1/\zeta(2s)$ to an accuracy of
$O(\{\zeta(2s)\}^3)$ on the interval $[1,\infty)$, where the approximation is exact at the end point $s=\infty$.

\begin{remark}
The B\'{a}ez-Duarte equivalence to the Riemann hypothesis, \cite{baez1}, \cite{maslanka}, using coefficients $c_t$ defined by
\be
c_t:=\sum_{s=0}^t(-1)^s\binom{t}{s}\frac{1}{\zeta(2s+2)},
\label{eq:n02}
\ee
asserts that the Riemann hypothesis is true if and only if for integers $t\geq 0$,
\be
c_t = O(t^{-3/4 + \epsilon}),\quad \text{for all } \epsilon >0.
\label{eq:n03}
\ee
Our approximation to $1/\zeta(2s)$ is probably not strong enough to use in the B\'{a}ez-Duarte equivalence to the Riemann Hypothesis in terms of re-stating the equivalence as sums of both $\zeta(2s)$ and $1/\zeta(2s)$.
\end{remark}

\section{Roots of the Ramanujan polynomials}
We conclude this paper with a brief look at the roots of the Ramanujan polynomials $R_r(z)$. In \cite{gun}, it was shown that $R_{2s+1}(z)$ is a polynomial in $z$ of degree $2s+2$ whose
four real roots are $z_0$, $1/z_0$, $-z_0$ and $-1/z_0$, where $z_0$ is the root of $R_{2s+1}(z)$ slightly greater than 2. It was also shown that the $2s-2$ complex roots of $R_{2s+1}(z)$ lie on the unit circle and as $s\rightarrow \infty$ the distribution of these nonreal roots on the unit circle becomes uniform. Specifically, the roots of unity that are zeros of $R_{2s+1}(z)$ are given by $\pm i$ when $s$ is even; all four of $\pm \rho$, $\pm \overline{\rho}$ when $s$ is a multiple of 3, and no others. Here $\rho$ is a cube root of unity.

In contrast, the even-indexed Ramanujan polynomials $R_{2s}(z)$ are of degree $2s$ in $z$ (as by Theorem 1.1 it can be seen that the leading terms cancel) and appear to only have the two real roots $\pm 1/2$, as detailed in Corollary 1 of Theorem 1.1. Explicit calculation suggests that for $s\geq 1$, $R_{2s}(z)$ has $2s-2$ complex roots, which all lie just outside the unit circle and whose distribution also becomes uniform as $s\rightarrow \infty$.

The zeros of the Ramanujan polynomials \cite{murty1} are important because they occur in expressions for the odd zeta values and as such the roots of $R_{2s}(z)$ may well be worth investigating further.


\small{School of Mathematics\\
Cardiff University\\
P.O. Box 926\\
Cardiff CF24 4AG\\
UK\\
E-Mail: LettingtonMC@cf.ac.uk;matt.lettington@sky.com}

\end{document}